\documentclass[11pt]{amsart}
\usepackage{amsmath}
\usepackage{amsthm}
\usepackage{amssymb}
\usepackage{amsfonts}
\usepackage{graphicx}
\usepackage{mathrsfs,bm,amscd}
\usepackage{latexsym,enumitem}
\usepackage{color}
\usepackage{marginnote}
\usepackage{hyperref}
\usepackage{caption}
\usepackage{subcaption}
\usepackage{graphicx}
\usepackage{mathrsfs,bm,amscd}

\usepackage[all]{xy}

\numberwithin{equation}{section}
\theoremstyle{plain}

\newtheorem{thm}{Theorem}[section]
\newtheorem{prop}[thm]{Proposition}
\newtheorem{cor}[thm]{Corollary}

\newtheorem{lem}[thm]{Lemma}

\theoremstyle{definition}

\newtheorem{definition}[thm]{Definition}

\newtheorem*{bprob*}{Bonus problem}

\def\R{\mathbb{R}}

\def\N{\mathbb{N}}

\def\e{\varepsilon}

\def\d{\delta}

\def\a{\alpha}

\DeclareMathOperator{\diam}{diam}
\DeclareMathOperator{\Hull}{Hull}
\DeclareMathOperator{\dist}{dist}

\newcommand{\cW}{\mathcal{W}}

\newcommand{\cE}{\mathcal{E}}

\newcommand{\eps}{\varepsilon}

\newcommand{\card}{\operatorname{card}}

\author{Efstathios-K. Chrontsios-Garitsis}
\author{Vyron Vellis}

\thanks{V.V was partially supported by NSF DMS grant 2154918.}

\address{Department of Mathematics\\ The University of Tennessee\\ Knoxville, TN 37966}
\email{echronts@utk.edu}
\address{Department of Mathematics\\ The University of Tennessee\\ Knoxville, TN 37966}
\email{vvellis@utk.edu}

\begin{document}

\title{Sobolev spaces on snowtrees}

\begin{abstract}
We introduce a discrete-energy Sobolev space $\cW^{1,p}_{\mathscr V}(T)$ on  Ahlfors regular snowtrees, a class of metric trees 
%where every two branches are bi-Lipschitz to each other uniformly
where every arc is a snowflake of the same type. Our main result shows that, for every partition $\mathscr V$ and every $1<p<\infty$, this discrete space coincides quantitatively with the Korevaar--Schoen space on $T$. This fact and the independence of the space on the particular partition used to define $\cW^{1,p}_{\mathscr V}(T)$ are both novel even for the class of geodesic trees. We also determine the critical Korevaar--Schoen exponent for Ahlfors regular snowtrees and prove  capacity attainment and upper estimates, which reveal the appropriate walk dimension needed for the corresponding probabilistic profile on these trees.
\end{abstract}

\maketitle

\section{Introduction}

An essential tool in the areas of PDEs and calculus of variations is the class of Sobolev functions. On Euclidean spaces they may be defined in several equivalent ways: by distributional derivatives, by upper gradients, by difference quotients, by variational energies, or by maximal functions. The increasing interest and  need to extend this theory to metric spaces soon emerged, and has been a very active line of research for the past two and a half decades.  Applications of this endeavor include the development of the theory of  PDEs \cite{KigamiAnFr}, calculus of variations \cite{AmbrosioBV} and optimal transportation \cite{AmbrosioGradFlow} on the non-smooth setting of fractal spaces. 

The theory of Sobolev-type functions defined on metric spaces has been developed by many authors (for instance \cite{CheegerSob}, \cite{HajSob}, \cite{HajKoskSob}, \cite{KorSchSob}, \cite{NagesNewtonianSob}). While  the equivalence of many of these viewpoints in the Euclidean setting is classical, in the metric setting the underlying space often determines the most appropriate notion, with different ideas needed to adjust the theory to different settings. We refer to \cite{Juha-Jeremy-etc-book} for a detailed exposition. It is also worth noting that the type of Sobolev functions a metric space supports reveals information on the appropriate probabilistic profile that it admits, especially through the relation of critical Sobolev exponents and the (random) walk dimension of the space \cite{JonssonCriticalExpWalkDim}. See also \cite{BarlowDiffusion, BrownianCarpet, TransCritExponFractals, MuruganRyoCarpet} for more connections to probability.

The metric spaces our manuscript focuses on is the class of (metric) trees. A metric space $X$ is called a \emph{(metric) tree} if it does not contain simple closed curves, and it is a Peano continuum, i.e., compact, connected, and locally  connected. In other words, any two points $x,y$ can be joined by a unique arc that has $x,y$ as endpoints. On the one hand, trees are among the most simple 1-dimensional connected metric spaces, but on the other hand they have infinitely many topological equivalence classes. 
%Nadler \cite{Nadler} proved that the class of metric trees contains a topologically universal element, i.e., there is a tree into which every tree embeds.
	
Recently, there has been great interest in a class of trees with controlled geometry, known as \emph{quasiconformal trees}, which play an important role in analysis on metric spaces. Quasiconformal trees were introduced by Kinneberg \cite{Kinneberg}\footnote{Kinneberg initially called these spaces ``quasi-trees'', while the term ``quasiconformal trees'' was first used by Bonk and Meyer in \cite{BM20}.} and have appeared in complex dynamics, dimension theory, geometric group theory, harmonic analysis, random processes, and geometric function theory. See \cite{CGIV} and the references therein for an extensive discussion on the importance of quasiconformal trees.
	
Two of the most well-known and used quasiconformal trees are the \emph{continuum self similar tree} (CSST) and the \emph{Vicsek fractal}. See Figure \ref{fig:CSST}. Both are attractors of iterated function systems of similarities in $\R^2$, but we omit their formal definitions. The CSST is a trivalent quasiconformal tree that has been studied in \cite{BT_CSST, BM22} and is almost surely homeomorphic to the \emph{(Brownian) continuum random tree} introduced by Aldous \cite{Aldous1,Aldous2}. The \emph{Vicsek fractal}, is a 4-valent metric tree of great importance in analysis, probability, and physics; see for instance \cite{Vicsek,Metz,HamblyMetz,Zhou,CSW,BaudChenVicsek}. The authors jointly with Ioannidis recently showed in \cite{CGIV} that the CSST and the  Vicsek fractal are quasisymmetrically universal spaces for certain sub-classes of quasiconformal trees, further emphasizing the need for the development of a suitable Sobolev framework that fits these spaces.
	
	\begin{figure}[h]
		\centering
		\begin{minipage}{0.55\textwidth}
			\centering
			\includegraphics[width=0.9\textwidth]{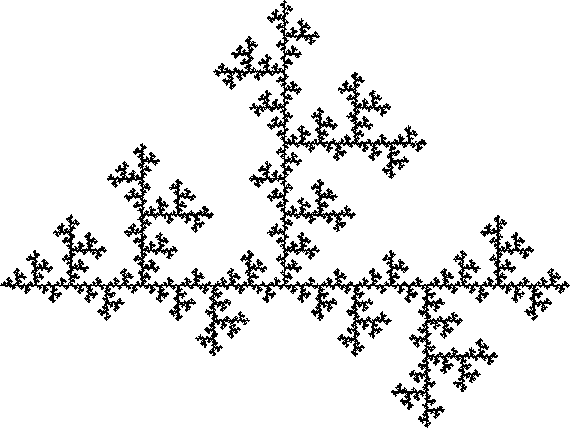}
		\end{minipage}\hfill
		\begin{minipage}{0.35\textwidth}
			\centering
			\includegraphics[width=0.9\textwidth]{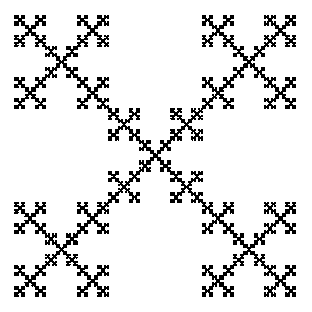}
		\end{minipage}
		\caption{The CSST (left) and the Vicsek fractal (right).}
		\label{fig:CSST}
	\end{figure}
	
The Korevaar--Schoen notion is a particularly well suited notion of Sobolev functions for the class of quasiconformal trees, since it does not depend on the existence of rectifiable arcs. Originally introduced in \cite{KorSchSob}, it measures the average oscillation of a function at all small scales. On the other hand, a different class that relies on a quotient-summability property of Sobolev functions is that of discrete energy-based Sobolev spaces. While this property aligns particularly well with the topology of metric trees, it has been particularly useful in establishing dimension distortion results in even more general   Euclidean  \cite{Kaufman, OurQCspec, QRdimDist, FraTysInterm} and  metric  \cite{CompHoldPAMS, AlvChronJFA} settings. Recent work has focused on the relation of the two, such as the work of Baudoin and Chen on the Vicsek fractal \cite{BaudChenVicsek}, the results of Murugan and Shimizu on the Sierpinski carpet and its conformal dimension \cite{MuruganRyoCarpet}, as well as the work of Anttila, Eriksson-Bique and Shimizu on Laakso-type fractal spaces \cite{SobolevLaaksoSylv}. The purpose of this paper is to carry out this program for a broad subclass of trees called \textit{snowtrees}. These trees have arcs which are all uniform snowflakes of the same type. 

\begin{definition}\label{defn:snowtrees}
A metric tree $(T,d)$ is an $\eps$-\emph{snowtree}, for some $\eps\in(0,1]$, if for any distinct $x,y\in T$, the arc $T[x,y]\subset T$ satisfies $\mathcal{H}^{1/\eps}(T[x,y])\simeq d(x,y)^{1/\eps}$.
\end{definition} 
Here and for the rest of this paper, we denote by $T[x,y]$ the unique closed arc in $T$ with endpoints $x$ and $y$. We use the corresponding notation for open and half-open arcs. It turns out that snowtrees are exactly those trees that are bi-H\"older equivalent to geodesic trees; see Appendix \ref{app}.

In order to define discrete energies, we need an appropriate discrete approximation of the underlying tree. 

\begin{definition}
A \emph{multiscale partition} of a tree $T$ is a nested family $\mathscr{V}=(V_n)_{n\in\N}$ of finite subsets of $T$ with the following properties.
\begin{enumerate}
\item The set $V_1$ contains only one element that is a leaf of $T$, denoted by $\rho_{\mathscr{V}}$ and called the \emph{root} of $\mathscr{V}$), and 
\[ \{\rho_{\mathscr{V}}\} = V_1 \subset V_2 \subset \cdots \subset T.\]
\item For each $n\in\N$, every branch point of $\Hull(V_n)$ is in $V_n$.
\item  We have that $\lim_{n\to\infty}\sup_{z\in T}\dist(z,V_n)=0$.
\end{enumerate}
Given a multiscale partition $\mathscr{V}=(V_n)_{n\in\N}$ of $T$, denote by $E_n$, for each $n\in\N$, the set of points $(x,y) \in V_n\times V_n$ such that $T[x,y]$ is the closure of a connected component of $\Hull(V_n)\setminus V_n$ and $x\in T[\rho_{\mathscr{V}},y]$.
\end{definition} 

Recall that given a closed subset $A$ of a tree $T$, the \emph{hull of $A$}, denoted by $\Hull(A)$ is the smallest connected subset of $T$ that contains $A$. Note that every tree has a multiscale partition, by considering appropriate level-nets.  %Note that $E_1=\emptyset$.

\begin{definition}\label{def:p-energyVS}
Let $T$ be an $\e$-snowtree, and denote by $C(T)$ the collection of continuous functions on $T$. Let $1 \leq p <+\infty$, $\mathscr{V}$ be a multiscale partition of $T$, $K\subset T$ be a continuum, and $f \in C(T)$. The discrete $p$-energy of $f$ on $K$ with respect to $\mathscr{V}$ is defined as
\[ \cE_{\mathscr{V},K}^p(f):=\sup_{n\in\N} \sum_{\substack{x,y \in V_n\cap K\\ (x,y)\in E_n}} \frac{|f(x)-f(y)|^p}{d(x,y)^{\tfrac{p-1}{\eps}}}. \]
Moreover, we denote by $\cW^{1,p}_{\mathscr{V}}(T)$ the collection of all functions $f\in C(T)$ with $\cE_{\mathscr{V},T}^{p}(f)<\infty$.
\end{definition}

Another way of defining Sobolev spaces on trees is via the \emph{Korevaar-Schoen} energy.

\begin{definition}
For $R>0$, $\a>0$, and $f\in C(T)$, define
$$E_{p,\a}(f,R):=\int_T\frac1{\mu(B(x,R))}\int_{B(x,R)}\frac{|f(y)-f(x)|^p}{R^{p\alpha}}\,d\mu(y)\,d\mu(x).$$
The \emph{Korevaar-Schoen(-Sobolev) space} on $T$ is defined as
$$ KS^{1,p}_{\alpha}(T):=\left\{f\in C(T):\limsup_{R\to0^+}E_{p,\a}(f,R)<\infty\right\}.$$
\end{definition}

In our main result we show that for a subclass of quasiconformal trees,  specifically for any Ahlfors regular $\e$-snowtree $T$, the space $\cW^{1,p}_{\mathscr{V}}(T)$ is a Banach space and independent of the choice of multiscale partition $\mathscr{V}$, as it coincides with $KS^{1,p}_{\alpha_p}(T)$ in a quantitative fashion for a specific value $\a_p$. The independence is new already for the case of the Vicsek fractal, which was studied in \cite{BaudChenVicsek} using a specific partition that heavily relies on the self-similarity of the space.

\begin{thm}\label{thm:discrete-KS-equivalence}
Let $(T,d,\mu)$ be a $Q$-Ahlfors regular $\eps$-snowtree, let $p\in (1,\infty)$, and let
$$\alpha_p := \frac Qp+\frac1\eps-\frac1{p\eps}.$$
For each multiscale partition $\mathscr{V}$ on $T$, we have
\begin{equation}\label{eq:mainthm}
\cE_{T,\mathscr{V}}^p(f)\simeq\limsup_{R\to0^+}E_{p,\a_p}(f,R),
\end{equation}
with constants depending only on $p$, $Q$, $\eps$ and the constant of Ahlfors regularity. In particular, $\cW^{1,p}_{\mathscr{V}}(T)=KS^{1,p}_{\alpha_p}(T)$.
\end{thm}

It should be noted that the above result is novel even in the case of $\eps=1$, and in particular in the case of geodesic trees. 

Our next result is the attainment of the Korevaar-Schoen critical exponent. This exponent plays a crucial role in Besov interpolation \cite{BaudChenVicsek, BaudLocalGeodTrees} and  in determining the proper probabilistic framework for the walk dimension of the underlying space (see for instance \cite{JonssonCriticalExpWalkDim, SobolevLaaksoSylv}).

\begin{thm}\label{thm: critical-KS-exponent}
Suppose $(T,d,\mu)$ is a $Q$-Ahlfors regular $\eps$-snowtree. Then
$$\alpha_p=\sup\left\{\alpha\ge0: KS^{1,p}_\alpha(T)\text{ contains a nonconstant function}\right\}.$$
Moreover, $KS^{1,p}_{\alpha_p}(T)$ contains a nonconstant function.
%In particular,
%\begin{enumerate}
%\item if $0\le \alpha\le \alpha_p$, then $KS^{1,p}_\alpha(T)$ contains a nonconstant function;
%\item if $\alpha>\alpha_p$, then every function in $KS^{1,p}_\alpha(T)$ is constant.
%\end{enumerate}
\end{thm}

Furthermore, we emphasize the resemblance of $p\alpha_p$ to a type of ``$p$-walk dimension'' by proving appropriate capacity estimates in the discrete energy setting, similarly to the work of \cite{SobolevLaaksoSylv} on Laakso-graphs. 

Let $(T,d)$ be a tree, let $x\in T$, let $\mathscr{V}$ be a multiscale partition of $T$, and let $r>0$, $A>1$, and $p\geq 1$. The discrete energy $p$-capacity of the annulus $B(x,Ar)\setminus B(x,r)$ is defined to be
\begin{align*}
\operatorname{Cap}_{p,\mathscr{V}}&(x,r,A)\\  
&:=\inf\left\{\cE_{\mathscr{V},T}^p(f): f\in\cW^{1,p}_{\mathscr{V}}(T),\ f|_{B(x,r)}=1,\ f|_{T\setminus B(x,Ar)}=0 \right\}.
\end{align*}

\begin{thm}\label{thm: upper cap bd}
Let $(T,d,\mu)$ be a $Q$-Ahlfors regular $\eps$-snowtree, let $\mathscr{V}$ be a multiscale partition, and let $1<p<\infty$. There exists $A>1$ such that for every $x\in T$ and $0<r<\infty$, 
%there exists $f\in\cW^{1,p}_{\mathscr{V}}(T)$ satisfying $f=1$ on $B(x,r)$, $f=0$ on $T\setminus B(x,Ar)$, and
%$$ \cE_{\mathscr{V},T}^p(f) \lesssim\frac{\mu(B(x,r))}{r^{p\alpha_p}}.$$
%In particular,
\begin{equation}\label{eq: cap upper estimate}
\operatorname{Cap}_{p,\mathscr{V}}(x,r,A)\lesssim\frac{\mu(B(x,r))}{r^{p\alpha_p}},
\end{equation} 
and the capacity is attained.
\end{thm}

Due to Theorem \ref{thm:discrete-KS-equivalence}, this also yields a similar upper estimate for the Korevaar-Schoen capacity; see Corollary \ref{cor: upper cap KS}. 

The paper is organized as follows. In Section~\ref{sec:discrete energies} we  prove an arc-wise derivative representation of the discrete energy, as well as a Morrey-Sobolev estimate. Section~\ref{sec: KS same discrete} proves the equivalence between discrete energy Sobolev spaces and Korevaar--Schoen spaces by comparing the respective energies. In Section \ref{sec: crit exp cap} we identify the critical Korevaar--Schoen exponent and prove the capacity upper bound and attainment. In Section \ref{sec:remarks} we list some open questions and further directions. Finally, in Appendix~\ref{app} we show that $\eps$-snowtrees are exactly the $\frac1{\e}$-bi-H\"older images of geodesic trees.

\subsection*{Acknowledgments} The authors wish to thank Sylvester Eriksson-Bique for the interesting discussions on the topic.

\subsection{Background and notation}\label{sec: backg}
A metric measure space $(X,d,\mu)$ is called \emph{$Q$-Ahlfors regular}, for some $Q>0$, if there is $C_A>0$ such that
\[C_A^{-1}R^Q \leq \mu(B(x,R)) \leq C_A R^{Q},\]
for all $x\in T$ and $0<R<\diam{T}$.

Given two non-negative quantities $A$ and $B$, we write $A\lesssim B$ if there is a \textit{comparability constant}  $C=C(\lesssim)$ such that $A\leq C B$. Similarly,  we write $A\gtrsim B$ if there is $C=C(\gtrsim)$ such that $A\geq  B/C$. If $A\lesssim B$ and $A\gtrsim B$ we write $A\simeq B$.

%Given $x,y\in T$ in a metric tree $(T,d)$, we denote by $T[x,y]$ the unique simple closed arc connecting $x$ to $y$, and we follow similar notation for half-open and open arcs.  

\section{Discrete energies on trees}\label{sec:discrete energies}

Fix for this section an $\e$-snowtree $(T,d)$, and a multiscale partition $\mathscr{V}=(V_n)_{n\in\N}$. The main goal of this section is to establish the following proposition. 

\begin{prop}\label{prop: discrete-gradient}
Let $p\in (1,\infty)$ and $f \in C(T)$. Then $f \in \cW^{1,p}_{\mathscr{V}}(T)$ if and only if there exists $g\in L^p(T,\mathcal{H}^{1/\e})$ such that for every $n \in \N$,
\begin{equation}\label{eq:FTC}
f(y)-f(x)=\int_{T[x,y]} g \,  d\mathcal{H}^{1/\e}, \quad\text{for every $(x,y)\in E_n$.}
\end{equation}
Moreover, if $K \subset T$ is connected and $f \in \cW^{1,p}_{\mathscr{V}}(T)$, then 
$$ \cE_{\mathscr{V},K}^p(f)\simeq \int_{K} |g|^p \, d\mathcal{H}^{1/\e}.$$
\end{prop}

Before we start the proof a few remarks are in order. First, the function $g$ need not be defined on the entire tree $T$ but rather on a subset of $T$ determined by $\mathscr{V}$. More precisely, set $S_n = \Hull(V_n)$ for each $n\in\N$, and note that $S_1\subset S_2 \subset \cdots \subset T$. Let $S = \bigcup_{n\in\N}S_n$, which we call the \textit{skeleton} of $T$ induced by $\mathscr{V}$. In general, $S$ may be a lot smaller than $T$; e.g. if $T$ is the Vicsek tree, for the particular multiscale partition considered in \cite{BaudChenVicsek} we have $\mathcal{H}^{\log{5}/\log{3}}(T)>0$ while $\mathcal{H}^{\log{5}/\log{3}}(S)=0$. With this notation, in Proposition \ref{prop: discrete-gradient}, we only require that $g\in L^p(S,\mathcal{H}^{1/\e})$.

Second, a more general version of Proposition \ref{prop: discrete-gradient} holds for all metric trees. More precisely, let $(T,d)$ be a metric tree. Assume that there exists a measure $\nu$ on the Borel $\sigma$-algebra $\mathcal{B}(T)$ such that 
\[0 < \nu(T(x,y)) < \infty, \quad\text{for all distinct $x,y\in T$}.\]
Note that $T$ possesses at least one such measure, due to Nadler's embedding theorem \cite{Nadler92}. More specifically, Nadler proved that every metric tree topologically embeds into a geodesic tree, and so $T$ is homeomorphic to a geodesic tree. Therefore, one can define $\nu$ to be the push-forward measure of $\mathcal{H}^1$ on the geodesic tree that is homeomorphic to $T$. If we define the discrete $p$-energy  of $f\in C(T)$ with respect to a multiscale partition $\mathscr{V}$ as
\[ \cE_{\mathscr{V},\nu,K}^p(f):=\sup_{n\in\N} \sum_{\substack{x, \,y \in V_n\cap K\\ (x,y)\in E_n}} \frac{|f(x)-f(y)|^p}{\nu(T[x,y])^{p-1}}, \]
then Proposition \ref{prop: discrete-gradient} holds verbatim with $\mathcal{H}^{1/\e}$ replaced by $\nu$.

\begin{proof}[Proof of Proposition \ref{prop: discrete-gradient}]
Given $n\in\N$ and a continuum $K\subset T$, let
\begin{equation}\label{eq:discreteenergy}
\cE_{V_n,K}^p(f):=\sum_{\substack{x, \,y \in V_n\cap K\\ (x,y)\in E_n}} \frac{|f(x)-f(y)|^p}{d(x,y)^{\tfrac{p-1}{\eps}}}. 
\end{equation}

Assume first that there exists $g\in L^p(T,\mathcal{H}^{1/\e})$ such that \eqref{eq:FTC} is true for all $n\in\N$. Let $K\subset T$ be a connected set, and let $n\in \N$. Applying \eqref{eq:FTC} and H\"older inequality, we have
\begin{align*} 
\cE_{V_n,K}^p(f) &=\sum_{\substack{x,y \in V_n\cap K\\ (x,y)\in E_n}} \frac{|f(x)-f(y)|^p}{d(x,y)^{\tfrac{p-1}{\eps}}}\\
&\lesssim \sum_{\substack{x,y \in V_n\cap K\\ (x,y)\in E_n}} \frac{(\int_{T[x,y]} |g| \,  d\mathcal{H}^{1/\e})^p}{d(x,y)^{\tfrac{p-1}{\eps}}}\\
&\leq \sum_{\substack{x,y \in V_n\cap K\\ (x,y)\in E_n}} \frac{(\mathcal{H}^{1/\e}(T[x,y]))^{p-1}\int_{T[x,y]} |g|^p \,  d\nu}{d(x,y)^{\tfrac{p-1}{\eps}}}.
\end{align*}
Recalling that $\mathcal{H}^{1/\eps}([x,y])\simeq d(x,y)^{1/\eps}$, that $S_n\subset S_{n+1}$, and the essential pair-wise disjointedness of arcs $T[x,y]$ for various $(x,y)\in E_n$, we obtain
$$\cE_{V_n K}^p(f)\lesssim\sum_{\substack{x,y \in V_n\cap K\\ (x,y)\in E_n}} \int_{T[x,y]} |g|^p \,  d\mathcal{H}^{1/\e}\leq \int_{S\cap K} |g|^p \, d\mathcal{H}^{1/\e}.$$
By taking supremum over all $n\in \N$ and using the hypothesis that $g\in L^p(S,\mathcal{H}^{1/\e})$, this direction is complete.

For the rest of the proof, assume that $f \in \cW^{1,p}_{\mathscr{V}}(T)$. Our goal is to construct a sequence of functions on $T$ with gradients in $L^p(S,\mathcal{H}^{1/\e})$ which converge to the desired $g$ in \eqref{eq:FTC}.

For each $n\in\N$, let $F_n: T\rightarrow\R$ be a piecewise constant function defined by
$$
F_n(z) =
\begin{cases}
\frac{f(x_{i+1})-f(x_{i})}{\mathcal{H}^{1/\e}(T[x_i,x_{i+1}])}, & \text{if $z\in T(x_i,x_{i+1})$ with $(x_i,x_{i+1})\in E_n$}\\
%w_z, & \text{if $z\in V_n$},\\
0, & \text{else}.
\end{cases}
$$ 
%where $w_z$ is chosen for every $z\in V_n$ as a fixed number from the set of values $\{F_n\vert_{a}: a\in A_n(z)\}$. 
Thus, we have a sequence of functions $(F_n)_n$ such that
\begin{align*}
\sup_n \|F_n\|^p_{L^p(K \cap S,\mathcal{H}^{1/\e})}
&= \sup_n \int_{K \cap S} |F_n|^p \, d\mathcal{H}^{1/\e} \\
&= \sup_n \int_{K \cap S_n} |F_n|^p \, d\mathcal{H}^{1/\e} \\
&\leq \sup_n \sum_{\substack{(x,y)\in E_n\\ T[x,y] \cap K \neq \emptyset}} \left| \frac{f(x) - f(y)}{\mathcal{H}^{1/\e}(T[x,y])} \right|^p \mathcal{H}^{1/\e}(T[x,y]) \\
&\lesssim  \mathcal{E}_{\mathscr{V},K}^p(f) \\
&< \infty.
\end{align*}
This implies that $(F_n)_n$ is a bounded sequence in $L^p(S,\mathcal{H}^{1/\e})$ which is a reflexive space. By Mazur lemma,  there is a sequence $(g_k)_k\subset L^p(S,\mathcal{H}^{1/\e})$ that converges to some $g\in L^p(S,\mathcal{H}^{1/\e})$ in the $p$-norm, with 
    \begin{equation}\label{eq: Mazur seq convex}
        g_k=\sum_{j=k}^{M_k} \lambda_j^k F_j, \quad \lambda_j^k\geq 0, \quad \sum_{j=k}^{M_k} \lambda_j^k=1.
    \end{equation}

We show that this $g$ is the desired function. Let $1\leq k \leq n$ be integers and let $(x,y)\in E_k$. Then
\begin{equation}\label{eq: int of gn}
\int_{T[x,y]} g_n\, d\mathcal{H}^{1/\e}=\int_{T[x,y]} \sum_{j=n}^{M_n} \lambda_j^n F_j \,d\mathcal{H}^{1/\e}=\sum_{j=n}^{M_n} \lambda_j^n\int_{T[x,y]}  F_j \,d\mathcal{H}^{1/\e}.
\end{equation} For $j\geq n$, since $S_n \subset S_j$, we can find points $x=q_0,q_1,\dots,q_{N}=y$ such that $(q_{i-1},q_{i})\in E_j$ for all $i\in\{1,\dots,N\}$. Therefore,
\begin{align*}
\int_{T[x,y]} F_j \,d\nu &= \sum_{i=1}^{N}\int_{T[q_{i-1},q_{i}]} F_j \, d\mathcal{H}^{1/\e}\\
&= \sum_{i=1}^{N}
    \frac{f(q_{i})-f(q_{i-1})}{\mathcal{H}^{1/\e}(T[q_{i-1},q_i])}
    \mathcal{H}^{1/\e}(T[q_{i-1},q_i])\\
&=\int_{T[x,y]} F_j \, d\mathcal{H}^{1/\e}\\
&=f(y)-f(x),\end{align*}
which along with \eqref{eq: int of gn} and \eqref{eq: Mazur seq convex} yields
$$\int_{T[x,y]} g_n\, d\mathcal{H}^{1/\e}=\sum_{j=n}^{M_n} \lambda_j^n(f(y)-f(x))=f(y)-f(x).$$ 
The $L^p$ convergence of $g_n$ to $g$, along with the above, completes the proof of \eqref{eq:FTC}.

It remains to show that $\int_{S\cap K}|g|^pd\nu\lesssim \cE_{\mathscr{V},K}^{p}(f)$. Indeed,
\begin{align*}
\int_{S\cap K} |g|^p \, dv
    &= \int_{S\cap K} \left|\lim_{n\to\infty} g_n \right|^p \, d\mathcal{H}^{1/\e} \\
    &= \int_{S\cap K} \left|\lim_{n\to\infty}
    \sum_{j=k}^{M_k} \lambda_j^k F_j \right|^p \, d\mathcal{H}^{1/\e} \\
    &\leq \int_{S\cap K} \left|\lim_{n\to\infty}
    \sup_{j\leq n} F_j \right|^p \, d\mathcal{H}^{1/\e} \\
    &\leq \sup_{j\in \mathbb{N}}
    \int_{S\cap K} |F_j|^p \, d\mathcal{H}^{1/\e} \\
    &\lesssim \cE_{\mathscr{V},K}^{p}(f),
\end{align*} 
where the last inequality is a result of the definitions of $F_j$ and $\cE_{\mathscr{V},K}^{p}(f)$.
\end{proof}

We end this section with a Morrey-type inequality that is essential to both the comparison to the Korevaar-Schoen space and the capacity estimate. 

\begin{prop}\label{prop: Morrey-Sobolev}
Let $K\subset T$ be closed and connected, $p>1$, and $f\in \cW^{1,p}_{\mathscr{V}}(T)$. Then for all $x,y\in K$ we have
\begin{equation*}
|f(x)-f(y)|^p\lesssim d(x,y)^{\tfrac{p-1}{\varepsilon}} \cE_{K,\mathscr{V}}^p(f).
    \end{equation*}
\end{prop}

\begin{proof}
Let $x,y\in (\bigcup_nV_n)\cap K$, and let $m\in \N$ be minimal so that $x,y\in V_m$. Suppose $x=q_0, q_1, \dots, q_M=y\in V_m$ are all the $m$-vertices lying on the arc $T[x,y]$. Since $K$ is connected, $x=q_0, q_1, \dots, q_M=y\in K$. By Proposition \ref{prop: discrete-gradient} and H\"older inequality,
\begin{align*}
|f(x)-f(y)| \leq &\sum_{i=1}^{M-1} |f(q_i)-f(q_{i+1})|\\
&\lesssim \sum_{i=1}^{M-1} \int_{T[q_i,q_{i+1}]} |g| \, d\mathcal{H}^{1/\e} \\
&= \int_{T[x,y]} |g| \, d\mathcal{H}^{1/\e}\\
&\lesssim \left(\int_{T[x,y]} |g|^p \, d\mathcal{H}^{1/\e}u\right)^{1/p}\mathcal{H}^{1/\e}(T[x,y])^{\tfrac{p-1}{p}}.
\end{align*}
Since $T$ is an $\eps$-snowtree, $\mathcal{H}^{1/\e}(T[x,y])\simeq d(x,y)^{1/\varepsilon}$, and another application of Proposition \ref{prop: discrete-gradient} yields
$$|f(x)-f(y)|\lesssim d(x,y)^{\tfrac{p-1}{\varepsilon p}}\cE_{\mathscr{V},K}^{p}(f)^{1/p}.$$ 
By continuity of $f$ and density of $\bigcup_nV_n$ in $T$, the proof is complete.
\end{proof}

\section{Korevaar--Schoen energies on Ahlfors regular snowtrees}\label{sec: KS same discrete}
The goal of this section is to give the proof of Theorem \ref{thm:discrete-KS-equivalence}. Throughout this section we fix a compact $Q$-Ahlfors regular $\eps$-snowtree $(T,d,\mu)$. Let 
\[R_0 = \frac1{100}\min\{1,\diam{T}\}\]
and let $C_T\geq 1$ and $C_A\geq 1$ be constants such that 
$$C_T^{-1} d(x,y)^{\frac1\eps}\leq \mathcal{H}^{1/\e}([x,y])\leq C_T d(x,y)^{\frac1\eps},$$ 
for all $x,y\in T$ and 
\[C_A^{-1}R^Q \leq \mu(B(x,R)) \leq C_A R^{Q}\]
for all $x\in T$ and $0<R<\diam{T}$.
We also fix a multiscale partition $\mathscr{V} = (V_n)_{n\in \N}$ of $T$. Recall the definition of skeleton $S = \bigcup_{n\in\N}S_n$ from Section \ref{sec:discrete energies}. 

The proof of \eqref{eq:mainthm} splits into two parts, Proposition \ref{prop:discrete-to-KS} and Proposition \ref{prop:KS-to-discrete}. 

\begin{prop}\label{prop:discrete-to-KS}
Let $1<p<\infty$. If $f\in\cW^{1,p}_{\mathscr{V}}(T)$, then $f\in KS^{1,p}_{\alpha_p}(T)$ and
$$\limsup_{R\to0^+}E_{p,\a_p}(f,R)\le C\cE_{\mathscr{V},T}^p(f).$$
The constant $C$ depends only on $p$, $Q$, $C_A$, and $\eps$.
\end{prop}

\begin{proof}
Let $R>0$, $x \in T$, and $y\in B(x,R)$. By Propositions \ref{prop: discrete-gradient} and \ref{prop: Morrey-Sobolev}, there exists $g\in L^p(S,\mathcal{H}^{1/\e})$
such that
$$
    |f(y)-f(x)|^p\lesssim
    R^{\frac1\eps(p-1)}
    \int_{B(x,R)\cap S}|g|^p\,d\mathcal{H}^{1/\e}.
$$
Using this estimate on $E_{p,\a_p}(f,R)$ gives
\begin{align}\label{eq: Phi upper}
E_{p,\a_p}(f,R) &\le \frac{R^{\frac1\e(p-1)}}{R^{p\alpha_p}} \int_T\frac1{\mu(B(x,R))}\int_{B(x,R)}\left(\int_{B(x,R)\cap S}|g|^p\,d\mathcal{H}^{1/\e}\right)d\mu(y)\,d\mu(x) \\
&=R^{-Q}\int_T\int_{B(x,R)\cap S}|g|^p\,d\mathcal{H}^{1/\e}\,d\mu(x), \notag
\end{align}
where we used Ahlfors regularity and the fact that $p\alpha_p=Q+\frac1\eps(p-1)$. We apply Fubini-Tonelli theorem to the nonnegative measurable function
$$(x,z)\mapsto |g(z)|^p \chi_{\{d(x,z)<R\}},$$
since $\mu$ is finite and $\mathcal{H}^{1/\e}$ is $\sigma$-finite. Thus, by Ahlfors regularity of $\mu$,
\begin{align*}
\int_T\int_{B(x,R)\cap S}|g|^p\,d\mathcal{H}^{1/\e}\,d\mu(x)&=\int_T \int_S |g(z)|^p\chi_{\{d(x,z)<R\}}\,d\mathcal{H}^{1/\e}(z)\,d\mu(x) \\
&=\int_S |g(z)|^p\int_T \chi_{\{d(x,z)<R\}}\,d\mu(x)\,d\mathcal{H}^{1/\e}(z) \\
&=\int_S |g(z)|^p\mu(B(z,R))\,d\mathcal{H}^{1/\e}(z) \\
&\lesssim R^Q\int_S|g|^p\,d\mathcal{H}^{1/\e}.
\end{align*}
Using the above on \eqref{eq: Phi upper} yields 
$$E_{p,\a_p}(f,R)\lesssim\int_S|g|^p\,d\mathcal{H}^{1/\e}=\cE_{\mathscr{V},T}^{p}(f).$$
Taking $\limsup_{R\to0^+}$ completes the proof.
\end{proof}

The reverse inequality in \eqref{eq:mainthm} requires more work. The first step is to create a partition of unity on $T$. This is done in three steps. In Lemma \ref{lem:arcwise-representatives}, we show that $1/\e$-H\"older functions on $T$ support a notion of gradient. In Lemma \ref{lem:snowtree-cutoff} we create appropriate ``bump functions'' on $T$, and in Lemma \ref{lem:snowtree-partition} we create the partition of unity.

\begin{lem}\label{lem:arcwise-representatives}
Let $h\in C(T)$ for which there exists $L<\infty$ such that for all $n\in \N$ and all $(x,y)\in E_n$,
\begin{equation}\label{eq:h Lip}
|h(x)-h(y)|\le L\, \mathcal{H}^{1/\e}(T[x,y]).
\end{equation}
Then there exists  $\partial h:T\to \R$ such that $|\partial h(z)|\le L$ for $\mathcal{H}^{1/\e}$-a.e.~$z\in S$, and 
\begin{equation}\label{eq: h one-dim derivative}
h(y)-h(x)=\int_{T[x,y]} \partial h\,d\mathcal{H}^{1/\e}
\end{equation}
for all $n\in\N$ and $(x,y)\in E_n$. Moreover, if $K$ is a finite union of closed arcs and $h$ is constant on each component of $T\setminus K$, then $\partial h=0$ for $\mathcal{H}^{1/\e}$-a.e.~point of $T\setminus K$.
\end{lem}

\begin{proof}
First, we show the following stronger version of \eqref{eq:h Lip}: for all $n\in\N$ and all $(x,y)\in E_n$, 
\begin{equation}\label{eq:h Lip2}
|h(z)-h(w)|\le L\, \mathcal{H}^{1/\e}(T[z,w]), \quad \text{for all $z,w\in T[x,y]$}.
\end{equation}
Indeed, fix $z,w\in T[x,y]$. Without loss of generality, we may assume that $z\in T[x,w]$. For each $k\geq n$ let $z_k$ be the point in $V_k\cap T(x,y)$ closest to $z$, and let $w_k$ be the point in $V_k\cap T[z,y]$ closest to $w$. There also exist, for each $k\geq n$, points $z_k =q_{k,0},q_{k,1},\dots, q_{k,N_k}=w_k$ in $T[x,y]\cap V_k$ so that $(q_{k,i-1},q_{k,i}) \in E_k$ for all $i\in\{1,\dots,N_k\}$. Then, applying \eqref{eq:h Lip} several times,
\begin{align*}
|h(z_k)-h(w_k)| \leq \sum_{i=1}^{N_k}|h(q_{k,i-1})-h(q_{k,i+1})| &\leq L\sum_{i=1}^{N_k}\mathcal{H}^{1/\e}(T([q_{k,i-1},q_{k,i}]))\\ 
&= \mathcal{H}^{1/\e}(T([z_k,w_k])).
\end{align*}
By density of $\bigcup_k V_k$ in $T$, $z_k\to z$, $w_k \to w$, and $\mathcal{H}^{1/\e}(T([z_k,w_k])) \to \mathcal{H}^{1/\e}(T([z,w]))$ as $k\to \infty$. Moreover, by continuity of $f$, $f(z_k)\to f(z)$ and $f(w_k)\to f(w)$ as $k\to \infty$, which completes the proof of \eqref{eq:h Lip2}.

Recall now that $\mathcal{H}^{1/\e}(T[u,v])\simeq d(u,v)^{\frac1\eps}$ for each $u,v\in T$. Given distinct points $x,y\in T$, there exists a unique map
$$\gamma_{x,y}:[0,\mathcal{H}^{1/\e}(T[x,y])]\to T[x,y] $$
such that $\gamma_{x,y}(0)=x$, $\gamma_{x,y}(\mathcal{H}^{1/\e}(T([x,y])))=y$, and
$$\mathcal{H}^{1/\e}(T[\gamma_{x,y}(s),\gamma_{x,y}(t)])=|s-t|, \quad \text{for all $s,t\in[0,\mathcal{H}^{1/\e}(T[x,y])]$}.$$
% say inverse of F:a\to[0,\mathcal{H}^{1/\e}(a)],  F(z):=\mathcal{H}^{1/\e}([x,z])

By \eqref{eq:h Lip2}, for each $n\in\N$ and $(x,y)\in E_n$, the function 
\[ h\circ\gamma_{x,y}: [0,\mathcal{H}^{1/\e}(T[x,y])]\to \R\] 
is $L$-Lipschitz so it is differentiable at $\mathcal{H}^1$-a.e.~$t$, with
\begin{equation}\label{eq: one dim h Lip deriv}
    |(h\circ\gamma_{x,y})'(t)|\le L.
\end{equation}
We follow the convention that $(h\circ\gamma_{x,y})'(t)$ denotes the derivative at $t$ where $h\circ\gamma_{x,y}$ is differentiable, and $0$ otherwise. 

Define $\partial h_1:T\to \R$ by
$$\partial h_1(z) =
    \begin{cases}
    (h \circ \gamma_{x,y})'(t), & \text{if $z=\gamma_{x,y}(t)$ with $(x,y)\in E_1$,}\\
    %, $t\in (0,\mathcal{H}^{1/\e}(T[x,y]))$,} \\
    0, & \text{else}.
    \end{cases}$$ 
Assuming, we have defined $\partial h_{n}$ for some $n\in\N$, define $\partial h_{n+1}:T\to \R$ with
$$\partial h_{n+1}(z) =
    \begin{cases}
    \partial h_{n}(z), & \text{if $z\in T[x,y]$ for some $(x,y)\in E_{n}$,} \\
    (h \circ \gamma_{x,y})'(t), & \text{if $z=\gamma_{x,y}(t)$, $(x,y)\in E_{n+1}$, and $T[x,y] \not\subset S_n$},\\
    0, & \text{else}.
    \end{cases}$$
Note that for all $n\in\N$, $\partial h_n |T\setminus S =0$ and that $\partial h_{n+1}|S_n=\partial h_{n}|S_n$, $\mathcal{H}^{1/\e}$-a.e.~on $S_{n}$. If for $z\in S$ there is some $n\in \N$ with $\partial h_n(z)\neq 0$, denote by $n_z$ the smallest such integer. Define 
$\partial h:T\to\R$ with
$$\partial h(z) =
    \begin{cases}
    \partial h_{n_z}(z), & \text{if $z\in S$ and there exists $n\in\N$ with $\partial h_{n}(z)\neq 0$}, \\
    0, & \text{else}.
    \end{cases}$$
The $\mathcal{H}^{1/\e}$-a.e.~bound $|\partial h|\le L$ is immediate by \eqref{eq: one dim h Lip deriv}.

Given $(x,y)\in E_n$, by the absolute continuity of $h\circ\gamma_{x,y}$ and the fact that $\partial h(z)=\partial h_n(z)$ for $\mathcal{H}^{1/\e}$-a.e.~$z\in T[x,y]$, we have
$$ h(y)-h(x) = \int_{0}^{\mathcal{H}^{1/\e}(T[x,y])]}(h\circ\gamma_{x,y})'(t)\,dt = \int_{T[x,y]} \partial h\,d\mathcal{H}^{1/\e}.$$

For the second claim of the lemma, suppose that $K\subset T$ is a finite union of closed arcs and that $h$ is constant on each component of $T\setminus K$. For every $(x,y)\in \bigcup_{n\in\N} E_n$ with $T[x,y]$ being contained in a component of $T\setminus K$, the derivative of $h\circ \gamma_{x,y}$ is
zero $\mathcal{H}^1$-a.e., and so $\partial h(z)=0$ for $\mathcal{H}^{1/\e}$-a.e.~$z\in T[x,y]$. The proof is complete by an application of \eqref{eq: h one-dim derivative}.
\end{proof}

\begin{lem}\label{lem:snowtree-cutoff}
There exists a constant $C_1\ge1$ depending only on $C_T$, $\eps$, $Q$, and $C_A$ such that for every $x\in T$ and every $0<R<R_0$, there exists a function $\Psi_x^R\in C(T)$ with the following properties:
\begin{enumerate}
\item $0\le \Psi_x^R\le1$;
\item $\Psi_x^R|_{B(x,C_1^{-1}R)}\ge1/2$ and $\Psi_x^R|_{T\setminus B(x,C_1R)}=0$;
\item there exists  $\partial\Psi_x^R:T\to \R$ such that, for every $n\in\N$ and every $(u,v)\in E_n$,
$$\Psi_x^R(v)-\Psi_x^R(u)=\int_{T[x,y]} \partial\Psi_x^R\,d\mathcal{H}^{1/\e},$$
and $|\partial\Psi_x^R(z)|\le C_1 R^{-\frac1\eps}$ for $\mathcal{H}^{1/\e}$-a.e.~$z\in T$;
\item the support of $\partial\Psi_x^R$ is contained in $S\cap\overline{B(x,C_1R)}$, and
$$\mathcal{H}^{1/\e}\left(\{z\in S:\partial\Psi_x^R(z)\neq0\}\right)\le C_1 R^{\frac1\eps}.$$
\end{enumerate}
%In particular,
%$$\int_S|\partial\Psi_x^R|^p\,d\mathcal{H}^{1/\e}\le C R^{-\frac1\eps(p-1)}.$$
\end{lem}

\begin{proof}
Fix $x\in T$, $R\in(0,R_0)$.

Let $ U(x,R)$ be the set of points $u\in T\cap \partial B(x,R)$ for which there exists $w\in T\setminus B(u,R)$ with $u\in T(x,w)$.
%is a connected component $T'$ of $T\setminus\{u\}$ not containing
%$x$ with $\diam{T'}\ge R$. In particular, if $S(x,R)=\{u\in T: d(x,u)=R\}$, then
%$$U(x,R)=\Bigl\{ u\in S(x,R):\,\, \exists \,w\in T\setminus B(u,R)\,\, \text{with}\,\, [u,w]\cap[u,x]=\{u\}\Bigr\}.$$
We claim that $\card U(x,R)\leq M$, for some $M\in \N$ depending only on $C_T$, $\e$, $Q$, and $C_A$. To see that, fix for any $u\in U(x,R)$, a point $w_u \in \partial B(u,R)$ with $u\in T(x,w)$. (This is possible by connectedness of $T$ and the definition of $ U(x,R)$.) Then
$$ \mathcal{H}^{1/\e}(T[x,w_u]) = \mathcal{H}^{1/\e}(T[x,u])+\mathcal{H}^{1/\e}(T[u,w_u]) \le 2C_T R^{\frac1\eps}, $$
and, hence, $d(x,w_u)\le (2C_T^2)^{\e}R$. If $u, v\in U(x,R)$ are distinct, then the arc $T[w_u,w_v]$ contains the two disjoint  arcs $T[w_u,u]$ and $T[v,w_v]$. Therefore
\begin{align*}
C_T d(w_u,w_v)^{\frac1\eps} \ge \mathcal{H}^{1/\e}(T[w_u,w_v]) &\ge \mathcal{H}^{1/\e}(T[w_u,u])+\mathcal{H}^{1/\e}(T[v,w_v])\\ 
&\ge 2C_T^{-1}R^{\frac1\eps}. 
\end{align*}
In particular, $d(w_u,w_v)\ge (2C_T^{-2})^{\e}R$. Thus, the points $\{w_u:u\in U(x,R)\}$ are $cR$-separated, with $c= (2C_T^{-2})^{\e} >0$, and lie in $B(x,C_1R)$. By Ahlfors regularity, the cardinality of
$ U(x,R)$ is bounded by a constant depending only on $C_T$, $\e$, $Q$, $C_A$.

For $u\in U(x,R)$ and $y\in T$, let $c(x,y,u)$ be the unique point satisfying
$$ T[x,y]\cap T[x,u]=T[x,c(x,y,u)],$$
and define $b_u:T\to \R$ with
$$b_u(y):=\mathcal{H}^{1/\e}([x,c(x,y,u)]).$$
We claim that
\begin{equation}\label{eq: b_u is 1-Lip}
|b_u(y_1)-b_u(y_2)| \leq \mathcal{H}^{1/\e}([y_1,y_2]).
\end{equation}
The proof of \eqref{eq: b_u is 1-Lip} is a case study on the relative position of $y$ with respect to $x$ and $u$. Given distinct points $p,q\in T$, let $B_p(q)$ is the connected component of $T\setminus \{p\}$ that contains $q$. If $y\in T[x,u]$, then $b_u(y)=\mathcal{H}^{1/\e}(T[x,y])$. If $y\notin T[x,u]$ and $y\notin B_u(x)$, then $b_u(y)=\mathcal{H}^{1/\e}(T[x,u])$. If $y\notin T[x,u]$ and $y\in B_u(x)\cap B_x(u)$, then $b_u(y) \in (0,\mathcal{H}^{1/\e}(T[x,u]))$. Finally, if $y\notin T[x,u]$, $y\in B_u(x)$, and $y\notin B_x(u)$, then $b_u(y)=0$.

Now set
$$\Psi_x^R(y):=1-\max_{u\in U(x,R)}\frac{b_u(y)}{\mathcal{H}^{1/\e}(T[x,u])},$$
with the convention that the maximum is $0$ if $U(x,R)=\emptyset$. By \eqref{eq: b_u is 1-Lip}, for each $u\in U(x,r)$,
\begin{align*}
\left| \frac{b_u(y_1)}{\mathcal{H}^{1/\e}(T[x,u])} - \frac{b_u(y_2)}{\mathcal{H}^{1/\e}(T[x,u])} \right| \leq \frac{\mathcal{H}^{1/\e}([y_1,y_2])}{\mathcal{H}^{1/\e}(T[x,u])} \leq C_T R^{-1/\e}\mathcal{H}^{1/\e}([y_1,y_2]).
\end{align*}
Hence, by
%Since $\mathcal{H}^{1/\e}(T[x,u])\ge C_T^{-1}R^{\frac1\eps}$ for every $u\in U(x,R)$, the maximum is taken over a controlled number of{\color{blue}$C_T R^{-\frac1\eps}$-Lipschitz} functions with respect to the arc measure. Hence, by \eqref{eq: b_u is 1-Lip} and
Lemma \ref{lem:arcwise-representatives}, there exists $\partial\Psi_x^R:S\to \R$ such that for all $n\in\N$ and $(u_1,u_2)\in E_n$,
$$\Psi_x^R(u_2)-\Psi_x^R(u_1) = \int_{T[u_1,u_2]} \partial\Psi_x^R\,d\mathcal{H}^{1/\e}$$
and $|\partial\Psi_x^R(z)|\le C_T R^{-\frac1\eps}$ for $\mathcal{H}^{1/\e}$-a.e.~$z\in T$.

Since $0\le b_u(y)\le\mathcal{H}^{1/\e}([x,u])$, we have $0\le\Psi_x^R\le1$. Fix $c_1>0$ small enough so that $C_T^2c_1^{\frac1\eps}\le 1/2$. If $d(x,y)\le c_1R$, then
$$ b_u(y)\le\mathcal{H}^{1/\e}(T[x,y])\le C_Tc_1^{\frac1\eps}R^{\frac1\eps} \le \frac12 C_T^{-1}R^{\frac1\eps} \le \frac12\mathcal{H}^{1/\e}(T[x,u])$$
for every $u\in U(x,R)$, and hence $\Psi_x^R(y)\ge1/2$. Also, choose $c_2\ge1$ large enough so that $C_T^{-2}c_2^{\frac1\eps}\ge2$. If $d(x,y)\ge c_2R$, let $u$ be the unique point in $T[x,y] \cap \partial B(x,R)$, then
\begin{align*}
d(u,y)^{\frac1\eps}&\ge C_T^{-1}\mathcal{H}^{1/\e}(T[u,y])\\
&= C_T^{-1}\left(\mathcal{H}^{1/\e}(T[x,y])-\mathcal{H}^{1/\e}(T[x,u])\right)\\ 
&\ge C_T^{-2}d(x,y)^{\frac1\eps}-R^{\frac1\eps}\\ 
&\ge (C_T^{-2}c_2^{\frac1\eps}-1)R^{\frac1\eps}\\
&\geq R^{\frac1\eps}
\end{align*}
and so $u\in U(x,R)$, and $b_u(y)=\mathcal{H}^{1/\e}([x,u])$. Therefore, $\Psi_x^R(y)=0$.
Since $y$ is arbitrary, we have $\Psi_x^R=0$ on $T\setminus B(x,c_2 R)$.

It remains to determine the support of $\partial\Psi_x^R$. For fixed $u$, the function $b_u$ is constant on each component of $T\setminus T[x,u]$. This implies that the function $y \mapsto \max_{u\in U(x,R)} b_u(y)/\mathcal{H}^{1/\e}(T[x,u])$
is constant on each component of
$$T\setminus \bigcup_{u\in U(x,R)}T[x,u].$$
By the last part of Lemma~\ref{lem:arcwise-representatives}, we have that $\partial\Psi_x^R=0$ for $\mathcal{H}^{1/\e}$-a.e.~point in $S\setminus\bigcup_{u\in U(x,R)}T[x,u]$. Since the cardinality of $U(x,R)$ is uniformly bounded, the number of arcs in this union is also uniformly bounded, and each has $\mathcal{H}^{1/\e}(T[x,u])\le C_T R^{\frac1\eps}$. Therefore, there exists $c_3$ depending only on $C_T$, $\e$, $Q$, and $C_A$ such that
$$\mathcal{H}^{1/\e}\left(\{z\in S:\partial\Psi_x^R(z)\neq0\}\right)\le c_3 R^{\frac1\eps}.$$
Moreover, if $z\in T[x,u]$ for some $u\in U(x,R)$, then
$$ d(x,z)^{\frac1\eps}\le C_T\mathcal{H}^{1/\e}(T[x,z])\le C_T\mathcal{H}^{1/\e}(T[x,u])\le C_T^2R^{\frac1\eps},$$
and hence the support of $\partial\Psi_x^R$ is contained in $\overline{B(x,c_2R)}$ for the aforementioned large $c_2$. 
%As a result, \[\int_S|\partial\Psi_x^R|^p\,d\mathcal{H}^{1/\e} \le C R^{-\frac1\eps(p-1)}. \qedhere\]
\end{proof}

\begin{lem}\label{lem:snowtree-partition}
There exists $C_2 \geq 1$ depending only on $C_T$, $\e$, $Q$, and $C_A$ such that for every
$0<R<R_0$, there exist finitely many points $x_i\in T$ and functions $\phi_i^R\in C(T)$ satisfying:
\begin{enumerate}
\item the balls $B_i:=B(x_i,C_2^{-1}R)$ cover $T$, and each point of $T$ belongs to at most $N=N(C_2)\in \N$ balls $B_i$, and the points $\{x_i\}$ are $C_2^{-1}R/2$-separated;
\item $0\le\phi_i^R\le1$ and $\sum_i\phi_i^R=1$ on $T$;
\item $\phi_i^R=0$ on $T\setminus B(x_i,C_2R)$;
\item for each $i$, there exists $\partial\phi_i^R:T\to \R$ with support contained in $S\cap \overline{B(x_i,C_2R)}$, such that, for every $n\in\N$ and every $(u,v)\in E_n$, we have
$$ \phi_i^R(v)-\phi_i^R(u) = \int_{T[u,v]}\partial\phi_i^R\,d\mathcal{H}^{1/\e}, $$
and
$$\int_S|\partial\phi_i^R|^p\,d\mathcal{H}^{1/\e}\le C_2 R^{-\frac1\eps(p-1)}.$$
\end{enumerate}
\end{lem}

\begin{proof}
Let $C_1\ge1$ be the constant from Lemma \ref{lem:snowtree-cutoff}. Fix 
\[ c<\min\{R_0,C_1^{-1}\}.\] 
Let ${x_i}$ be a maximal $(cR/2)$-separated subset of $T$. By compactness of $T$, this set is finite. By maximality, the balls $B(x_i,cR)$ cover $T$. Moreover, if $x\in T$ is an arbitrary point, by Ahlfors regularity and the $(cR/2)$-separation of $x_i$'s, there exists $N\in\N$ depending only on $C_T$, $\e$, $Q$, and $C_A$ such that $x$ is contained in at most $N$ many balls $B(x_i,cR/2)$. This yields (1).

%Let$$\Psi_i:=\Psi_{x_i}^R$$
Recall the functions $\Psi_x^R$ defined in Lemma \ref{lem:snowtree-cutoff}. Since $c<C_1^{-1}$ and the union of the balls $B(x_i,cR)$ covers $T$, Lemma \ref{lem:snowtree-cutoff}(2) gives
\begin{equation}\label{eq: H sum def low bound}
H(u):=\sum_j\Psi_{x_j}^R(u)\ge\frac12,
\end{equation}
for all $u\in T$. Define
$$\phi_i^R:=\frac{\Psi_{x_i}^R}{H}:T\to\R. $$
Then (2) and (3) follow directly by Lemma \ref{lem:snowtree-cutoff}(1),(3).

It remains to show (4). Let $n\in \N$ and $(u,v)\in E_n$. By Lemma \ref{lem:snowtree-cutoff}(4) we have that
$$|\Psi_{x_i}^R(v)-\Psi_{x_i}^R(u)|\leq C_1R^{-\frac1\eps}\mathcal{H}^{1/\e}(T[u,v]),$$
for all $i$. This, Lemma \ref{lem:snowtree-cutoff}(1), and \eqref{eq: H sum def low bound} imply that
\begin{align*}
|\phi_i^R(u)-\phi_i^R(v)|
&= \left|\frac{\Psi_{x_i}^R(u)}{H(u)}- \frac{\Psi_{x_i}^R(v)}{H(v)}\right|\\
&\leq  \frac{|\Psi_{x_i}^R(u)-\Psi_{x_i}^R(v)|}{H(u)}+\frac{|H(v)\Psi_{x_i}^R(v)-H(u)\Psi_{x_i}^R(v)|}{H(u)H(v)}\\
&\leq 2C_1 R^{-\frac1\eps}\mathcal{H}^{1/\e}(T[u,v])+4|\Psi_{x_i}^R(v)||H(v)-H(u)|.
\end{align*}
However, by another application of Ahlfors regularity, along with (1), only at most finitely many terms are non-zero in the sums $H(u), H(v)$, with the upper bound on the terms depending only on $C_T$, $\e$, $Q$, and $C_A$. This and the above yield that
$$ |\phi_i^R(u)-\phi_i^R(v)|\lesssim R^{-\frac1\eps} \mathcal{H}^{1/\e}(T[u,v]),$$ 
which by Lemma \ref{lem:arcwise-representatives} and choice of $\phi_i^R$ are enough to complete the proof of (4), using the same arguments as in the proof of Lemma \ref{lem:snowtree-cutoff}(5).
\end{proof}

The next proposition completes the proof of Theorem \ref{thm:discrete-KS-equivalence}.

\begin{prop}\label{prop:KS-to-discrete}
Let $1<p<\infty$. If $f\in KS^{1,p}_{\alpha_p}(T)$, then $f\in\cW^{1,p}(T)$ and
$$ \cE_{T,\mathscr{V}}^p(f) \le C\liminf_{R\to0^+}E_{p,\a_p}(f,R).$$
The constant $C$ depends only on $p$, $C_T$, $Q$, $C_A$, and $\eps$.
\end{prop}

\begin{proof}
Set $L:=\liminf_{R\to0^+}E_{p,\a_p}(f,R)$, and choose a sequence $R_k\to0$ such that
$$\lim_{k\to \infty}E_{p,\a_p}(f,C^*R_k)=L,$$
where 
\[ C^* = 2C_2 + 2.\]
We may assume that $R_k<R_0$ for every $k$.

Let $\{\phi_i^{R_k}\}_i$ be the partition of unity from Lemma \ref{lem:snowtree-partition}. Write
$$B_i^k:=B(x_i^k,C_2^{-1}R_k), \quad f_i^k:=\frac1{\mu(B_i^k)}\int_{B_i^k}f\,d\mu,$$
and define
\begin{equation}\label{eq: def f_k in Prop KS-discrete}
    f_k:=\sum_i f_i^k\phi_i^{R_k}.
\end{equation}

For each $i$, let $\partial\phi_i^{R_k}$ be the function from Lemma \ref{lem:snowtree-partition}(4). In particular, for every $(x,y)\in E_n$ we have
$$\phi_i^{R_k}(y)-\phi_i^{R_k}(x) =\int_a \partial\phi_i^{R_k}\,d\mathcal{H}^{1/\e},$$
and 
$$\partial\phi_i^{R_k}=0, \quad \text{$\mathcal{H}^{1/\e}$-a.e.~on $S\setminus \overline{B(x_i^k,C_2R_k)}$.}$$
Define $h_k : S \to \R$ by
$$ h_k:=\sum_i f_i^k\,\partial\phi_i^{R_k}.$$
Since the sum is finite, for every $(x,y)\in E_n$ we have
\begin{align*}
\int_{T[x,y]} h_k\,d\mathcal{H}^{1/\e} = \sum_i f_i^k\int_{T[x,y]}\partial\phi_i^{R_k}\,d\mathcal{H}^{1/\e}  &= \sum_i f_i^k\bigl(\phi_i^{R_k}(y)-\phi_i^{R_k}(x)\bigr) \\
&= f_k(y)-f_k(x).
\end{align*}
In other words, $h_k$ behaves like a ``derivative'' for $f_k$.

We now show that $f_k\to f$ uniformly. If
$\phi_i^{R_k}(x)\neq0$, then $x\in B(x_i^k,C_2R_k)$ by Lemma \ref{lem:snowtree-partition}(3). Hence, $B_i^k\subset B(x,2C_2R_k)$. Therefore,
\begin{align*}
|f_i^k-f(x)| &= \left|\frac1{\mu(B_i^k)}\int_{B_i^k}\bigl(f(u)-f(x)\bigr)\,d\mu(u)\right| \\
&\le \frac1{\mu(B_i^k)} \int_{B_i^k}|f(u)-f(x)|\,d\mu(u) \\
&\le \sup\{|f(u)-f(x)|:\ u\in B_i^k\} \\
&\le \sup\{|f(u)-f(v)|:\ d(u,v)\le 4C_2R_k\}.
\end{align*}
Since $f$ is uniformly continuous on $T$, the above with \eqref{eq: def f_k in Prop KS-discrete} and Lemma \ref{lem:snowtree-partition}(2) gives
$$\|f_k-f\|_{L^\infty(T)}\to0.$$

We now estimate $h_k$. 
%Fix $k$ and ease the notation in the following by setting $R=R_k$, $x_i=x_i^k$, $B_i=B_i^k$, and $\phi_i=\phi_i^{R_k}$. Then,
%$$h_k=\sum_i f_i^k\,\partial\phi_i.$$
Set
$$D_{k}:=\sum_i\partial\phi_i^{R_k}.$$
For every $(x,y)\in E_n$ we have
\begin{align*}
\int_{T[x,y]} D_{k}\,d\mathcal{H}^{1/\e} = \sum_i \int_{T[x,y]}\partial\phi_i^{R_k}\,d\mathcal{H}^{1/\e} =\sum_i\bigl(\phi_i^{R_k}(y)-\phi_i^{R_k}(x)\bigr)&= 1-1=0.
\end{align*}
Thus, for every $n\in \N$ and every $(x,y)\in E_n$, $D_k$ is $\mathcal{H}^{1/\e}$-a.e.~equal to $0$ on $T[x,y]$. Since $S$ is a countable union of such arcs, we have that $D_k=0$  $\mathcal{H}^{1/\e}$-a.e.~on $S$.

Now let $j\neq i$, and  $\widetilde B_j^k:=B(x_j^k,C_2R_k)$. Since $D_k=0$, we get for $\mathcal{H}^{1/\e}$-a.e.~$z\in\widetilde B_j^k\cap S$ that
$$ h_k(z) = \sum_i f_i^k\,\partial\phi_i^{R_k}(z) = \sum_i(f_i^k-f_j^k)\partial\phi_i^{R_k}(z). $$
Moreover, if $z\in \widetilde B_j^k\cap S$ and $\partial\phi_i^{R_k}(z)\neq0$, then, after a potential modification on a null set, $z\in \overline{B(x_i^k,C_2R_k)}$. Hence, $d(x_i^k,x_j^k)\le 2C_2R_k$, implying that for $\mathcal{H}^{1/\e}$-a.e.~$z\in \widetilde B_j^k$, only indices $i$ with $d(x_i^k,x_j^k)\le 2C_2R_k$ contribute to the sum $h_k(z)$. By Lemma \ref{lem:snowtree-partition}(1), there are only at most uniformly many such indices. Therefore, by this and Lemma \ref{lem:snowtree-partition}(4), there exists $M_1$ depending only on $C_T$, $\e$, $Q$, and $C_A$ such that
\begin{align*}
\int_{\widetilde B_j^k\cap S}|h_k|^p\,d\mathcal{H}^{1/\e}  
&\le M_1 \sum_{i:\,d(x_i^k,x_j^k)\le 2C_2R_k} |f_i^k-f_j^k|^p \int_S|\partial\phi_i^{R_k}|^p\,d\mathcal{H}^{1/\e} \\
&\le C_2 M_1 R_k^{-\frac1\eps(p-1)} \sum_{i:\,d(x_i^k,x_j^k)\le 2C_2R_k} |f_i^k-f_j^k|^p .
\end{align*}
Since the balls $\widetilde B_j^k$ cover $T$, summing over all $j$ gives
\begin{equation}\label{eq: intS of h_k sums}
\int_S|h_k|^p\,d\mathcal{H}^{1/\e} \le M_1C_2 R_k^{-\frac1\eps(p-1)} \sum_j \sum_{i:\,d(x_i^k,x_j^k)\le 2C_2R_K} |f_i^k-f_j^k|^p .
\end{equation}

Let  $i,j$ be indices with $d(x_i^k,x_j^k)\leq 2C_2R_k$. By choice of $f_i^k, f_j^k$, and by Jensen's inequality and Ahlfors regularity
we have
\begin{align*}
|f_i^k-f_j^k|^p &\le \frac1{\mu(B_i^k)\mu(B_j^k)} \int_{B_i^k}\int_{B_j^k}|f(u)-f(v)|^p\,d\mu(u)\,d\mu(v) \\
&\le C_A^2 C_2^{2Q}R_k^{-2Q}\int_{B_i^k}\int_{B_j^k}|f(u)-f(v)|^p\,d\mu(u)\,d\mu(v).
\end{align*}
Using the above on \eqref{eq: intS of h_k sums} yields
\begin{equation}\label{eq: intS h_k sum Int}
\begin{aligned}
\int_S |h_k|^p\,d\mathcal{H}^{1/\e}
&\le M_1C_2^{1+2Q}C_A^2 R_k^{-\frac1\eps(p-1)-2Q}\\
&\qquad  \sum_j \sum_{i:\, d(x_i,x_j)\le 2C_2R_k} \int_{B_i^k}\int_{B_j^k} |f(u)-f(v)|^p\,d\mu(u)\,d\mu(v).
\end{aligned}
\end{equation}
If $u\in B_i^k$, $v\in B_j^k$, and $d(x_i^k,x_j^k)\le2C_2R_k$, then 
$$
d(u,v)\leq (2+2C_2)R_k = C^*R_k.
$$ Moreover, since the points $x_i^k, x_j^k$ are $C_2^{-1}R_k$ separated, by an application of Ahlfors regularity, for every fixed pair $(u,v)$ there are only at most uniformly many pairs $(i,j)$ such that $u\in B_i^k$, $v\in B_j^k$, and $d(x_i^k,x_j^k)\le 2C_1R$. As a result, there exists $M_2$ depending only on $C_T$, $\e$, $Q$, and $C_A$ such that
$$\sum_j\sum_{i:\,d(x_i^k,x_j^k)\le 2C_2R_k} \chi_{B_i^k}(u)\chi_{B_j^k}(v) \le M_2\chi_{B(v,C^*R_k)}(u).$$
By the above and an application of Fubini-Tonelli on \eqref{eq: intS h_k sum Int} we have
\begin{align*} 
\int_S&|h_k|^p\,d\mathcal{H}^{1/\e}\\ 
&\le M_3 R_k^{-\frac1\eps(p-1)-2Q} \int_T\int_{B(v,C^*R_k)} |f(u)-f(v)|^p\,d\mu(u)\,d\mu(v)\\
&\leq M_4 \frac1{R_k^{p\a_p}} \int_T \frac{1}{\mu(B(v,C^*R_k))} \int_{B(v,C^*R_k)} |f(u)-f(v)|^p\,d\mu(u)\,d\mu(v)\\
&= M_4E_{p,\a_p}(f,C^*R_k)
\end{align*}
recalling that $p\alpha_p=Q+\frac1\eps(p-1)$ and with $M_4=C_A(C^*)^QM_3$. Thus, since $E_{p,\a_p}(f,C^*R_k)\to L$, the sequence $(h_k)_k$ is bounded in $L^p(S,\mathcal{H}^{1/\e})$.

Since $1<p<\infty$, after passing to a subsequence we may assume that
$$h_k\rightharpoonup g\quad\text{weakly in $L^p(S,\mathcal{H}^{1/\e})$.}$$
Let $n\in \N$ and $(x,y)\in E_n$. For every $k\in \N$,
$$f_k(y)-f_k(x)=\int_{T[x,y]} h_k\,d\mathcal{H}^{1/\e}.$$
Since $f_k\to f$ uniformly and $p>1$, taking $k\to \infty$ yields
$$f(y)-f(x)=\int_a g\,d\mathcal{H}^{1/\e}.$$
Thus, by Proposition \ref{prop: discrete-gradient} we have that $f\in\cW^{1,p}(T)$ with
$$ \cE_{\mathscr{V},T}^{p}(f)\simeq\int_S|g|^p\,d\mathcal{H}^{1/\e}. $$
Finally, the weak lower semicontinuity of the $L^p$ norm gives
\begin{align*}
\cE_{\mathscr{V},T}^{p}(f) \simeq \int_S|g|^p\,d\mathcal{H}^{1/\e} \le \liminf_{k\to\infty}\int_S|h_k|^p\,d\mathcal{H}^{1/\e} &\le M_4\liminf_{k\to\infty}E_{p,\a_p}(f,C^*R_k) \\
&= M_4\liminf_{R\to0^+}E_{p,\a_p}(f,R).
\end{align*}
This completes the proof.
\end{proof}

%\begin{proof}[{Proof of Theorem \ref{thm:discrete-KS-equivalence}}]
%The implication $$\cW^{1,p}(T)\subset KS^{1,p}_{\alpha_p}(T)$$ and the estimate $$\limsup_{R\to0^+}\Phi_f(R)\le C\cE_{T,p}(f)$$ are due to Proposition~\ref{prop:discrete-to-KS}. The implication $$KS^{1,p}_{\alpha_p}(T)\subset \cW^{1,p}(T)$$ and the estimate $$\cE_{T,p}(f)\le C\liminf_{R\to0^+}\Phi_f(R)$$ are due to Proposition~\ref{prop:KS-to-discrete}. Combining the two estimates proves the equivalence of the spaces and the comparability of the energies.
%\end{proof}

\section{The critical Korevaar--Schoen exponent and capacity estimates}\label{sec: crit exp cap}
The goal of this section is to show that the number $\alpha_p$ is the critical Korevaar--Schoen exponent for any $Q$-Ahlfors regular $\eps$-snowtree $T$. Recall that the critical Korevaar--Schoen of $T$ is the largest regularity exponent $\alpha$ for which the Korevaar--Schoen space $KS^{1,p}_{\alpha}(T)$ contains non-constant continuous functions. 

Throughout this section we keep the assumptions and notation of Section~\ref{sec: KS same discrete}. In particular, $(T,d,\mu)$ is a $Q$-Ahlfors regular $\eps$-snowtree, with  $Q\geq 1$, and
$$\alpha_p =\frac Qp+\frac1\eps-\frac1{p\eps}.$$

 \begin{proof}[Proof of Theorem \ref{thm: critical-KS-exponent}]
Fix a multiscale partition $\mathscr{V}$ for $T$. Fix two distinct points $u,v\in T$ and set $e=T[u,v]$. For $z\in T\setminus e$, let
$\pi_e(z)$ be the unique point of $e$ such that
$$T[z,\pi_e(z)]\cap e=\{\pi_e(z)\}.$$
This point exists and is unique since $T$ is a tree and has no simple loops. For $z\in e$ set $\pi_e(z)=z$. Define $h:T\to \R$ by
$$h(z):=\mathcal{H}^{1/\e}(T[u,\pi_e(z)]).$$
Then $h(u)=0$ and $h(v)=\mathcal{H}^{1/\e}(T[u,v])>0$, so $h$ is non-constant.

Suppose that $\pi_e(z)\ne \pi_e(w)$. Then $z$ and $w$ lie on distinct connected components of $T\setminus e$. This implies that 
$$T[z,w]= T[z,\pi_e(z)]\cup T[\pi_e(z),\pi_e(w)]\cup T[\pi_e(w),w],$$
with the arcs on the right-hand size being essentially disjoint, i.e., the pairwise intersections contain at most one point. Hence, $T[\pi_e(z),\pi_e(w)]\subset T[z,w]$.
Therefore, 
\begin{equation}\label{eq:projection-Lip-nu}
|h(z)-h(w)|\le \mathcal{H}^{1/\e}(T[z,w]).
\end{equation} 
Note that if $\pi_e(z)= \pi_e(w)$, then the above inequality is trivially true and, thus, \eqref{eq:projection-Lip-nu} holds for all $z,w\in T$. In addition, since $\mathcal{H}^{1/\e}(T[z,w])=d(z,w)^{1/\e}$, we have by \eqref{eq:projection-Lip-nu} that $h\in C(T)$.

By \eqref{eq:projection-Lip-nu} and Lemma~\ref{lem:arcwise-representatives}, there exists a function $\partial h:S\to\R$ such that
$$h(y)-h(x)=\int_{T[x,y]}\partial h\,d\mathcal{H}^{1/\e},$$
for every $n\in\N$ and $(x,y)\in E_n$, and
$$|\partial h|\le 1 \quad \text{$\mathcal{H}^{1/\e}$-a.e.~on $S$}.$$
Moreover, $h$ is constant on each component of $T\setminus e$. Hence by Lemma~\ref{lem:arcwise-representatives}, $\partial h=0$ for $\mathcal{H}^{1/\e}$-a.e.~point of $S\setminus e$, which yields
$$\int_S|\partial h|^p\,d\mathcal{H}^{1/\e} = \int_{S\cap e}|\partial h|^p\,d\mathcal{H}^{1/\e} \le \mathcal{H}^{1/\e}(e) = d(u,v)^{\frac1\eps} <\infty.$$
By Proposition~\ref{prop: discrete-gradient}, the above implies that $h\in\cW^{1,p}_{\mathscr{V}}(T)$. As a result, by Theorem~\ref{thm:discrete-KS-equivalence} we also have that the non-constant function $h$ lies in
$KS^{1,p}_{\alpha_p}(T)$.

Let $0\le\alpha<\alpha_p$. Note that for every $R>0$, we have
$$ E_{p,\a}(h,R) = R^{p(\alpha_p-\alpha)}E_{p,\a_p}(h,R).$$
Since $\limsup_{R\to0^+}E_{p,\a_p}(h,R)<\infty$, it follows that
$$\limsup_{R\to0^+}E_{p,\a}(h,R)<\infty.$$
Thus $h\in KS^{1,p}_\alpha(T)$, and so $KS^{1,p}_\alpha(T)$ contains a
nonconstant function for every $0\le\alpha\le\alpha_p$.

On the other hand, let $\alpha>\alpha_p$ and let $f\in KS^{1,p}_\alpha(T)$. Then
$$E_{p,\a_p}(f,R) =R^{p(\alpha-\alpha_p)}E_{p,\a}(f,R).$$
Since $\limsup_{R\to0^+}E_{p,\a}(f,R)<\infty$ and $\alpha-\alpha_p>0$, we obtain
$$\limsup_{R\to0^+}E_{p,\a_p}(f,R)=0.$$
In particular, $f\in KS^{1,p}_{\alpha_p}(T)$, and by Theorem \ref{thm:discrete-KS-equivalence} we have
$$ \cE_{\mathscr{V},T}^p(f) \simeq \liminf_{R\to0^+}E_{p,\a_p}(f,R) = 0.$$
Applying Proposition~\ref{prop: Morrey-Sobolev} with $K=T$, we get for all $x,y\in T$ that
$$|f(x)-f(y)|^p \lesssim d(x,y)^{\tfrac{p-1}{\eps}}\cE_{\mathscr{V},T}^p(f) =0.$$
Thus, $f$ is constant. Since $f\in KS^{1,p}_\alpha(T)$ is arbitrary, this proves that for no exponent $\alpha>\alpha_p$ there are non-constant functions in $KS^{1,p}_\alpha(T)$.
\end{proof}

The rest of this section is devoted to the proof of Theorem \ref{thm: upper cap bd}.

\begin{proof}[Proof of Theorem \ref{thm: upper cap bd}]
Let $R_0$ be as in the beginning of Section \ref{sec: KS same discrete} and $C_1\ge1$ be the constant from Lemma~\ref{lem:snowtree-cutoff}. Let 
\begin{equation}\label{eq: cap A choice}
    A\ge \max\left\{C_1^2,\frac{C_1\diam T}{R_0} \right\} \geq 1.
\end{equation}
We first prove the result for $0<r<C_1^{-1}R_0$.
Set $R:=C_1r$, and let $\Psi:=\Psi_x^R$ be the function given by
Lemma~\ref{lem:snowtree-cutoff}. Thus, $0\le \Psi\le1$, $\Psi|_{B(x,r)}\geq 1/2$ and $\Psi|_{B(x,C_1^2r)}=0$. Let $\eta:\R\to\R$ be such that
%the $2$-Lipschitz function
$$ \eta(t):=
\begin{cases}
0, & t\le0,\\
2t, & 0<t<\frac12,\\
1, & t\ge \frac12.
\end{cases}$$
Define $f:=\eta\circ\Psi$. Then $f\in C(T)$, $0\le f\le1$, $f=1$ on $B(x,r)$, and $f=0$ on $T\setminus B(x,C_1^2r)$.

It remains to estimate the discrete energy of $f$. Let $\partial\Psi:S\to\R$ be the function from Lemma~\ref{lem:snowtree-cutoff}. Set $G_R:=\{z\in S:\partial\Psi(z)\ne0\}$. By Lemma~\ref{lem:snowtree-cutoff},
\begin{equation}\label{eq: cap supp meas}
\mathcal{H}^{1/\eps}(G_R)\le CR^{1/\eps }
\end{equation}
and
\begin{equation}\label{eq: cap deriv bd}
|\partial\Psi|\le C_1 R^{-1/\eps } \quad \text{$\mathcal{H}^{1/\e}$-a.e. on $S$}.
\end{equation}
Fix $n\in\N$ and $(u,v)\in E_n$. Since $\eta$ is $2$-Lipschitz, Lemma~\ref{lem:snowtree-cutoff} and  \eqref{eq: cap deriv bd} yield
\begin{align*}
|f(u)-f(v)| \le 2|\Psi(u)-\Psi(v)| &\le 2\int_{T[u,v]} |\partial\Psi|\,d\mathcal{H}^{1/\e}\\
&\le 2C_1 R^{-1/\eps }\mathcal{H}^{1/\eps}(T[u,v]\cap G_R).
\end{align*}

Recall the discrete energies $\mathcal{E}_{V_n,T}^p$ from \eqref{eq:discreteenergy}. Since $T$ is a snowtree, $\mathcal{H}^{1/\eps}(T[u,v])\simeq d(u,v)^{1/\eps }$, which along with the above implies
\begin{align*}
\cE_{V_n,T}^p(f) &= \sum_{(x,y)\in E_n} \frac{|f(x)-f(y)|^p}{d(x,y)^{(p-1)/\eps}}\\
&\lesssim \sum_{(x,y)\in E_n} \frac{|f(x)-f(y)|^p}{\mathcal{H}^{1/\eps}(T[x,y])^{p-1}}\\
&\lesssim R^{-p/\eps} \sum_{(x,y)\in E_n} \frac{\mathcal{H}^{1/\eps}(T[x,y]\cap G_R)^p}{\mathcal{H}^{1/\eps}(T[x,y])^{p-1}}\\
&\le R^{-p/\eps } \sum_{(x,y)\in E_n}\mathcal{H}^{1/\eps}(T[x,y]\cap G_R).
\end{align*}

Hence, because distinct arcs $T[x,y],T[x',y']$ with $(x,y),(x',y')\in E_n$ intersect only at endpoints, which are countably many, the above and \eqref{eq: cap supp meas} yield
$$\cE_{V_n,T}^p(f) \le R^{-p/\eps}\mathcal{H}^{1/\eps}(G_R) \lesssim R^{-p/\eps}R^{1/\eps} = R^{-(p-1)/\eps}.$$
The implicit constants are independent of $n$. Taking the supremum over
$n\in\N$ and recalling that $R=C_1r$ gives
\begin{equation}\label{eq: cap small}
\cE_{\mathscr{V},T}^p(f)\lesssim R^{-(p-1)/\eps} \lesssim r^{-(p-1)/\eps}.
\end{equation}
In particular $f\in\cW^{1,p}_{\mathscr{V}}(T)$.
Since 
$$p\alpha_p=Q+(p-1)/\eps,$$
and $\mu$ is $Q$-Ahlfors regular, \eqref{eq: cap small} yields
$$\cE_{\mathscr{V},T}^{p}(f) \lesssim \mu(B(x,r))r^{-p\alpha_p}.$$
This completes the proof for $0<r<C_1^{-1}R_0$, due to \eqref{eq: cap A choice}.

The case $C_1^{-1}R_0\le r<\infty$ remains. Note that due to \eqref{eq: cap A choice}, we have $Ar\ge \diam T$ for all $r\ge C_1^{-1}R_0$. Then $T\subset B(x,Ar)$ for every $x\in T$ and every $r\ge C_1^{-1}R_0$. In this case take $f\equiv1$. Then $f|_{B(x,r)}=1$, $f=0$ on $T\setminus B(x,Ar)=\emptyset$, and $\cE_{\mathscr{V},T}^{p}(f)=0$. Therefore, the desired estimate is trivially true.

It remains to show that the infimum in $\operatorname{Cap}_{p,\mathscr{V}}(x,r;A)$ is attained. Fix $x\in T$, $r>0$. The attainment of $\operatorname{Cap}_{p,\mathscr{V}}(x,r;A)$ is trivial for $r\ge C_1R_0$, so suppose $r<C_1 R_0$. Let $(f_j)_j$ be a sequence of functions $f_j:T\to \R$ so that
$$f_j|_{B(x,r)}=1,\quad f_j|_{T\setminus B(x,Ar)}=0,$$ 
for all $j$, and $\cE_{\mathscr{V},T}^{p}(f_j)\to \operatorname{Cap}_{p,\mathscr{V}}(x,r;A)$. After potentially passing to a subsequence, we may assume that for every $j$ we have
$$\cE_{\mathscr{V},T}^{p}(f_j)\le 2\operatorname{Cap}_{p,\mathscr{V}}(x,r;A).$$

We may further reduce to the case $0\le f_j\le1$. Indeed, let $\Tilde{\eta}:\R\to \R$ be the $1$-Lipschitz function with
$$\Tilde{\eta}(t):=
\begin{cases}
0, & t\le0,\\
t, & 0<t<1,\\
1, & t\ge1.
\end{cases}$$
Hence, for every $n\in\N$ and every
arc $(y_1,y_2)\in E_n$,
$$ |\Tilde{\eta}(f_j(y_1))-\Tilde{\eta}(f_j(y_2))| \le |f_j(y_1)-f_j(y_2)|,$$
which implies
$$ \cE_{V_n,T}^{p}(\Tilde{\eta}\circ f_j) \le \cE_{V_n,T}^{p}(f_j),$$
for all $n\in \N$, and so
$$ \cE_{\mathscr{V},T}^{p}(\Tilde{\eta}\circ f_j) \le \cE_{\mathscr{V},T}^{p}(f_j).$$
Moreover, $\Tilde{\eta}\circ f_j$ satisfies $\Tilde{\eta}\circ f_j=1$ on
$B(x,r)$ and $\Tilde{\eta}\circ f_j=0$ on $T\setminus B(x,Ar)$. By replacing $f_j$ by $\Tilde{\eta}\circ f_j$, we may therefore assume that $0\le f_j\le1$ on $T$ for every $j$.

By Proposition~\ref{prop: Morrey-Sobolev}, we have for all $y,z\in T$ that
$$|f_j(y)-f_j(z)|^p \lesssim d(y,z)^{(p-1)/\eps}\cE_{\mathscr{V},T}^{p}(f_j) \lesssim d(y,z)^{(p-1)/\eps} \operatorname{Cap}_{p,\mathscr{V}}(x,r;A),$$
Thus, by the above and by \eqref{eq: cap upper estimate} we have that $(f_j)_j$ is equicontinuous. Due to $0\le f_j\le1$, by
the Arzela--Ascoli theorem and after potentially passing to a subsequence we ahve that $f_j\to u$ uniformly on $T$, for some  $u\in C(T)$.
The uniform convergence gives
\begin{equation}\label{eq: cap u is admis}
u|_{B(x,r)}=1 \quad \text{and} \quad u|_{T\setminus B(x,Ar)}=0.
\end{equation}

It remains to show that $u\in\cW^{1,p}_{\mathscr{V}}(T)$ and that it minimizes the capacity. Fix $n\in\N$. Since $\cE_{V_n,T}^{p}$ is a finite sum, the uniform convergence implies
$$ \cE_{V_n,T}^p(u)=\lim_{j\to\infty}\cE_{V_n,T}^p(f_j)\leq \lim_{j\to\infty}\cE_{\mathscr{V},T}^p(f_j)=\operatorname{Cap}_{p,\mathscr{V}}(x,r;A).$$
Taking the supremum over $n$ the above yields
$$ \cE_{\mathscr{V},T}^p(u)\le \operatorname{Cap}_{p,\mathscr{V}}(x,r;A).$$
In particular, $u\in\cW^{1,p}_{\mathscr{V}}(T)$, and due to \eqref{eq: cap u is admis} we also have
$$\operatorname{Cap}_{p,\mathscr{V}}(x,r;A)\le \cE_{\mathscr{V},T}^{p}(u),$$
completing the proof.
\end{proof}

Similarly, one can define the Korevaar-Schoen $p$-capacity to be
\begin{align*}
&\operatorname{Cap}_{p,KS}(x,r,A)\\  
&:=\inf\left\{\limsup_{R\to0^+}E_{p,\a_p}(f,R): f\in KS^{1,p}(T),\ f|_{B(x,r)}=1,\ f|_{T\setminus B(x,Ar)}=0 \right\}.
\end{align*}

Theorem \ref{thm:discrete-KS-equivalence} in conjunction with Theorem \ref{thm: upper cap bd} yield the following corollary.

\begin{cor}\label{cor: upper cap KS}
Let $(T,d,\mu)$ be a $Q$-Ahlfors regular $\eps$-snowtree and let $1<p<\infty$. There exists $A>1$ such that for every
$x\in T$ and $0<r<\infty$, 
%there exists $f\in KS^{1,p}(T)$ satisfying $f=1$ on $B(x,r)$, $f=0$ on $T\setminus B(x,Ar)$, and
$$
\operatorname{Cap}_{p,KS}(x,r,A)\lesssim\frac{\mu(B(x,r))}{r^{p\alpha_p}}.
$$
\end{cor}

\section{Final remarks}\label{sec:remarks}

The equivalence of the two Sobolev notions in Theorem \ref{thm:discrete-KS-equivalence} in the context of Ahlfors regular snowtrees paves the way for various other directions. One natural follow-up would be based on our manuscript and the work of Baudoin, Chen and Yang in \cite{BaudLocalGeodTrees}. Namely, one of their results was a  characterization of the Korevaar-Schoen functions on Ahlfors regular locally geodesic trees, and utilizing that to prove heat kernel estimates and interpolation results. It would be natural to adjust our arguments for the proof of Theorem \ref{thm:discrete-KS-equivalence} in the setting of Ahlfors regular  \textit{locally} snowtrees and build a similar probabilistic and interpolation framework in this setting.

Perhaps a more ambitious direction would be to aim for a Sobolev equivalence similar to that of Theorem \ref{thm:discrete-KS-equivalence} on general quasiconformal trees $T$. Various parts of our arguments relied on the fact that all arcs are uniformly $\eps$-snowflakes, as well as the Ahlfors regularity of the whole space. We expect some partial results to follow using our methods, such as for Ahlfors regular trees whose arcs are $\delta$-snowflakes with $\delta\in \{\eps_1,\dots, \eps_m\}$, for some fixed $m\in \N$. Moreover, we expect it to be possible to also weaken the Ahlfors regularity to a slightly different volume-control condition and still attain partial comparability results in terms of the discrete energy and Korevaar-Schoen spaces. However, to address such a compatibility of spaces in a more general setting of quasiconformal trees different methods would have to be employed.

Last but not least, it would be interesting to investigate what other Sobolev notions coincide with the discrete energy Sobolev functions on Ahlfors regular snowtrees. For instance, what Haj{\l}asz-Sobolev spaces would coincide, if any, with $\cW^{1,p}(T)$? In the case of the Vicsek fractal, it is shown in \cite{BaudChenVicsek} that certain fractional Haj{\l}asz-Sobolev spaces are essentially trivial, but there is no mention of other exponents which could provide equivalent classes to the discrete energy or Korevaar-Schoen class.

\appendix

\section{Snowtrees are snowflakings of geodesic trees}\label{app}

We give an equivalent definition for snowtrees. Recall that given $\d>0$, a homeomorphism $f : (X,d_X) \to (Y,d_Y)$ is \emph{$\d$-bi-H\"older} if
\begin{equation*}
L^{-1} d_X(x,y)^{\d} \leq d_Y(f(x),f(y)) \leq L d_X(x,y)^{\d}, \quad\text{for all $x,y\in X$}.
\end{equation*}
%given a metric space $(X,d)$ and $\delta\in (0,1]$, the pair $(X,d^\d)$ (called the \emph{$\d$-snowflaking} of $(X,d)$) is also a metric space. 
We show in the next proposition that $\e$-snowtrees are bi-H\"older equivalent to geodesic trees; in particular, 1-snowtrees are bi-Lipschitz equivalent to geodesic trees. A similar result for arcs instead of trees can be found in \cite{HerMa}.

\begin{prop}\label{prop:snow1}
A tree $T$ is an $\eps$-snowtree if and only if there exists a geodesic tree $T'$ and an $\e$-bi-H\"older homeomorphism $f: T' \to T$.
%it is bi-Lipschitz equivalent to the $\e$-snowflaking of a geodesic tree.
\end{prop}

\begin{proof}
We first assume that $(T,d)$ is an $\eps$-snowtree. Define $\rho:T\times T\to \R$ with $\rho(x,y)=\mathcal{H}^{1/\eps}(T[x,y])$ for all $x,y\in T$. To show that $\rho$ is a metric, it is enough to prove the triangle inequality for distinct $x,y,z\in T$. There are four cases to consider.

First, suppose $T[x,y]\subset T[x,z]$. Then 
%by definition of $\rho$ and the Hausdorff measure,
\begin{align*}
\rho(x,z) = \mathcal{H}^{1/\eps}(T[x,z]) = \mathcal{H}^{1/\eps}(T[x,y]) + \mathcal{H}^{1/\eps}(T[y,z]) = \rho(x,y) + \rho(y,z).
\end{align*}
Second, suppose $T[x,y]\cap T[x,z]=\{x\}$. Then 
\begin{align*}
\rho(x,z)&= \mathcal{H}^{1/\eps}(T[x,z]) \\
&\leq \mathcal{H}^{1/\eps}(T[x,z]) + \mathcal{H}^{1/\eps}(T[x,y]) \\
&= \mathcal{H}^{1/\eps}(T[y,z]) \\
&= \rho(y,z) \\
&\leq \rho(x,y) + \rho(y,z).
\end{align*}
Third, suppose $T[x,y]\cap T[x,z]=T[x,z]$. Then
$$\rho(x,z)= \mathcal{H}^{1/\eps}(T[x,z])\leq \mathcal{H}^{1/\eps}(T[x,y])=\rho(x,y)\leq \rho(x,y)+\rho(y,z).$$
Fourth, suppose $T[x,y]\cap T[x,z]=T[x,y']$ for some $y'\neq y, z$, since these cases have been addressed. Then
\begin{align*}
\rho(x,z)&=\mathcal{H}^{1/\eps}(T[x,y'])+\mathcal{H}^{1/\eps}(T[y',z]),\\
\rho(x,y)&=\mathcal{H}^{1/\eps}(T[x,y'])+\mathcal{H}^{1/\eps}(T[y',y]),\\
\rho(y,z)&=\mathcal{H}^{1/\eps}(T[y,y'])+\mathcal{H}^{1/\eps}(T[y',z]).
\end{align*}
These yield that $\rho(x,z)\leq \rho(x,y)+\rho(y,z)$.
As a result, $(T,\rho)$ is indeed a metric space.

Moreover, since $(T,d)$ is a $\eps$-snowtree, we have that
\begin{equation}\label{eq: snowbiLip pf eq}
\rho(x,y)\simeq d(x,y)^{1/\eps}, \quad\text{for all $x,y\in T$.}
\end{equation}
This implies that $F:(T,d)\to (T,\rho)$ with $F(x)=x$ for all $x\in T$ is a homeomorphism and, thus, $(T,\rho)$ is a metric tree. In addition, if 
\[ F_\eps:(T,d)\to (T,\rho^\eps), \quad\text{$F_\eps(x)=x$ for all $x\in T$,}\] 
then \eqref{eq: snowbiLip pf eq} implies that $F_\eps$ is a bi-Lipschitz map. 
    
It remains to show that $(T,\rho)$ is geodesic. Let $x,y\in T$ be distinct. Let $\gamma:[0,1]\to T[x,y]$ be a parametrization  of the arc $T[x,y]$. Then
$$ \ell(T[x,y])=\sup\left\{ \sum_{i=1}^n\rho(\gamma(t_{i-1}),\gamma(t_{i})): \,n\in \N,\,\, 0=t_0<\dots<t_n=1  \right\}.$$ 
However, for any $n\in \N$ and any partition $t_i$, we have
$$\sum_{i=1}^n\rho(\gamma(t_{i-1}),\gamma(t_{i}))=\sum_{i=1}^n \mathcal{H}^{1/\eps}(T[\gamma(t_{i-1}),\gamma(t_{i})])=\mathcal{H}^{1/\eps}(T[x,y])=\rho(x,y).$$ 
This completes the proof of the more challenging direction.    

For the opposite direction assume that there exists an $(1/\e)$-bi-H\"older homeomorphism $F:(T,d)\to (T',\rho)$ where $(T',\rho)$ is a geodesic tree. For the rest of the proof we denote by $\diam_d$ and $ \diam_{\rho}$ the diameters with respect to the metrics $d$ and $\rho$, respectively, and we follow a similar notation for the Hausdorff measure. For any non-empty $U\subset T$, we have
\begin{align*}
	\diam_d U &= \sup\{d(p,q):\, p,q\in U\}\\
	&\simeq \sup\{\rho(F(p),F(q))^\eps:\, p,q\in U\}\\ 
	&= (\diam_{\rho} F(U))^\eps, 
%	&= \diam_{\rho^\eps} F(U),
\end{align*}
which by definition of the Hausdorff measure yields for all $x,y\in T$
$$\mathcal{H}^{1/\eps}_d(T[x,y])\simeq \mathcal{H}^1_\rho(T[F(x),F(y)])=\rho(F(x),F(y))\simeq  d(x,y)^{1/\eps}.$$ This completes the proof of the remaining direction.
\end{proof}

\end{document}